\documentclass[11pt,reqno]{amsart}
\usepackage{amsmath,amsthm,amssymb}

\newtheorem{theorem}{Theorem}[section]

\newtheorem{proposition}[theorem]{Proposition}

\theoremstyle{definition}

\newtheorem{remark}[theorem]{Remark}

\newtheorem{example}[theorem]{Example}

\numberwithin{equation}{section}

\begin{document}

\title{Samelson complex structures for the tangent Lie group}

\author{David N. Pham}
\address{Department of Mathematics $\&$ Computer Science, QCC CUNY, Bayside, NY 11364}
\curraddr{}
\email{dpham90@gmail.com}
\thanks{This work was supported by a QCC CUNY Fellowship Leave Award}



\dedicatory{}

\begin{abstract}
It is shown that for any compact Lie group $G$ (odd or even dimensional), the tangent bundle $TG$ admits a left-invariant integrable almost complex structure, where the Lie group structure on $TG$ is the natural one induced from $G$.  The aforementioned complex structure on $TG$ is inspired by Samelson's construction for even dimensional compact Lie groups.  
\end{abstract}

\date{}

\maketitle

\section{Introduction}
Let $G$ be a Lie group.  An almost complex structure on $G$ is an endomorphism $J:TG\rightarrow TG$ of the tangent bundle satisfying $J^2=-id$. This condition implies that $G$ must be even dimensional.  From the
Newlander-Nirenberg theorem, $J$ arises from a complex manifold structure on $G$ if and only if
$$
N_J(X,Y):=J[JX,Y]+J[X,JY]+[X,Y]-[JX,JY]=0
$$
for all $X,Y\in \mathfrak{X}(G)$.  One easily verifies that $N_J$ is a skew-symmetric tensor.  ($N_J$ the the so-called Nijenhuis tensor.) 

Since a Lie group is also an algebraic object, one is primarily interested in structures which are (at least) left-invariant.  An almost complex structure $J$ is left-invariant if
$$
J\circ(l_g)_\ast=(l_g)_\ast\circ J
$$
for all $g\in G$, where $l_g: G\rightarrow G$, $x\mapsto gx$ is left translation by $g$ and $(l_g)_\ast: TG\rightarrow TG$ is the natural pushforward on the tangent bundle. Let $\mathfrak{g}:=\mbox{Lie}(G)$ denote the space of left-invariant vector fields on $G$. Then left-invariance of $J$ is equivalent to the statement that $JX\in \mathfrak{g}$ for all $X\in \mathfrak{g}$.  Since $N_J$ is a tensor and a basis of $\mathfrak{g}$ forms a global frame of $G$, it follows that $N_J=0$ if and only if $N_J(X,Y)=0$ for all $X,Y\in \mathfrak{g}$. An almost complex structure which gives rise to a complex manifold structure (or equivalently satisfies $N_J=0$) is called an integrable almost complex structure.

Identifying $\mathfrak{g}$ with the tangent space $T_eG$ at the identity, we see that the problem of finding a left-invariant integrable almost complex structure on $G$ is purely algebraic.  One looks for a linear map $J:\mathfrak{g}\rightarrow \mathfrak{g}$ satisfying
$$
J^2=-id,\hspace*{0.2in} N_J(X,Y)=0\hspace*{0.1in}\forall X,Y\in \mathfrak{g}.
$$
Since the problem can be reduced to the Lie algebra level, we will also call a linear map $J: \mathfrak{g}\rightarrow \mathfrak{g}$ an integrable almost complex structure if it satisfies the above two conditions.

Of course, it is the condition $N_J=0$ which poses the challenge. In fact, there are Lie groups which admit no left-invariant integrable almost complex structures.  For example, it was shown in \cite{GR2002} that the Lie algebra with basis $x_1,\dots, x_{2n}$ ($n\ge 2$) and nonzero brackets
$$
[x_1,x_i]=x_{i+1},\hspace*{0.1in} 1<i<2n
$$
admits no integrable almost complex structures.  

Interestingly, it was shown by Samelson \cite{Sam1953} (and independently by Wang \cite{Wang1954} in the more general setting of homogenenous manifolds) that every even dimensional \textit{compact} Lie group always admits a left-invariant integrable almost complex structure.  The aforementioned complex structure is obtained by exploiting the properties of the root space decomposition of an even dimensional compact Lie group.  It is natural to ask if the Samelson construction can be adapted to certain classes of non-compact Lie groups.  Taking inspiration from \cite{Sam1953}, we show that the tangent bundle of any compact Lie group $G$ (odd or even dimensional) admits a left-invariant integrable almost complex structure, where the Lie group structure on $TG$ is the natural one induced from $G$.  $TG$ with this Lie group structure is called the \textit{tangent Lie group}.    

The rest of the paper is organized as follows.  In Section 2, we give a brief review of the tangent Lie group and the root space decomposition of a compact Lie group.  In Section 3, we show that the tangent Lie group of any compact Lie group (odd or even dimensional) always admits a left-invariant integrable almost complex structure.  We conclude the paper in Section 4 with some examples.

\section{Preliminaries}
In this section, we review some of the relevant background in an effort to make the paper reasonably self-contained while simultaneously establishing our notation.
\subsection{The tangent Lie group}
Let $m: G\times G\rightarrow G$ and $\iota:G\rightarrow G$ denote the group multiplication and group inverse respectively.  Identifying 
$$
T_{(g,h)} G\times G \simeq T_gG\times T_h G,\hspace*{0.1in}\forall~(g,h)\in G\times G,
$$
the natural Lie group structure on $TG$ is given by the pushforward of $m$ and $\iota$:
$$
m_\ast: TG\times TG\rightarrow TG,\hspace*{0.1in}\iota_\ast: G\rightarrow G.
$$
For $X\in \mathfrak{g}=T_eG$ and $g\in G$, let $X_g:=(l_g)_\ast X$ and let $\varphi: TG\rightarrow G\times \mathfrak{g}$ be the vector space isomorphism (and, in particular, diffemorphism) given by 
$$
\varphi(X_g):=(g,X).
$$
From this point forth, we identify $TG$ with $G\times \mathfrak{g}$ via $\varphi$.  Transferring the group structure from $TG$ to $G\times \mathfrak{g}$ so that $\varphi$ becomes a Lie group isomorphism, we see that
\begin{equation}
    \label{EqTGmult}
    (g,X)\cdot (h,Y)=(gh, \mbox{Ad}_{h^{-1}}X+Y)
\end{equation}
and 
\begin{equation}
    \label{EqTGinverse}
    (g,X)^{-1}=(g^{-1},-\mbox{Ad}_{g}X),
\end{equation}
where $\mbox{Ad}: G\rightarrow \mbox{GL}(\mathfrak{g})$ denotes the adjoint representation.  Let $\mathcal{G}:=\mbox{Lie}(TG)$ denote the Lie algebra of $TG$.  As a \textit{vector space}, $\mathcal{G}=\mathfrak{g}\oplus \mathfrak{g}$.  However, from (\ref{EqTGmult}) and (\ref{EqTGinverse}), its Lie algebra structure is not the direct sum of Lie algebras.  Using (\ref{EqTGmult}) and (\ref{EqTGinverse}), one finds that the Lie bracket on $\mathcal{G}$ is given by
\begin{equation}
    \label{EqTGbracket}
    [(X_1,X_2),(Y_1,Y_2)] = ([X_1,Y_1],[X_1,Y_2]+[X_2,Y_1])
\end{equation}
for $X_i,Y_i\in \mathfrak{g}$, $i=1,2$.  Setting
$$
X^c:=(X,0),\hspace*{0.1in} X^v:=(0,X),
$$
(\ref{EqTGbracket}) can be summarized by the following rules:
\begin{equation}
    \label{EqTGcvRules}
    [X^c,Y^c]=[X,Y]^c,\hspace*{0.1in}[X^c,Y^v]=[X,Y]^v,\hspace*{0.1in}[X^v,Y^v]=0
\end{equation}
for $X,Y\in \mathfrak{g}$.  Using the notion of complete and vertical lifts of vector fields from a manifold to its tangent bundle, one sees that the left-invariant vector fields on $TG$ corresponding to $X^c$ and $X^v$ are respectively the complete and vertical lifts of the left-invariant vector field associated to $X$ on $G$ (see \cite{Yano1973} for details).

\subsection{Root space decomposition for compact Lie groups} If $G$ is a compact Lie group, then $G$ always admits a positive definite inner product which is invariant with respect to the adjoint representation $\mbox{Ad}: G\rightarrow \mbox{GL}(\mathfrak{g})$.  For example, if $\mu$ is a left-invariant volume form on $G$ satisfying 
$$
\int_G \mu =1
$$
and $\beta$ is any positive definite inner product on $\mathfrak{g}$, then one obtains an $\mbox{Ad}$-invariant positive definite inner product $\langle\cdot,\cdot\rangle$ via
$$
\langle X,Y\rangle :=\int_G \beta(\mbox{Ad}_{g^{-1}}(X),\mbox{Ad}_{g^{-1}}(Y))\mu_g
$$
for $X,Y\in \mathfrak{g}$.  If $\beta$ is already $\mbox{Ad}$-invariant to begin with, then the above construction simply returns $\beta$.  $\mbox{Ad}$-invariance then implies invariance with respect to the adjoint representation $\mbox{ad}: \mathfrak{g}\rightarrow \mathfrak{gl}(\mathfrak{g})$, $\mbox{ad}_XY:=[X,Y]$. Explicitly, $\mbox{ad}$-invariance of $\langle\cdot,\cdot \rangle$ means that 
$$
\langle{\mbox{ad}_XY,Z}\rangle =-\langle{Y,\mbox{ad}_XZ}\rangle
$$
for $X,Y,Z\in \mathfrak{g}$.  The above condition implies that $\mbox{ad}_X: \mathfrak{g}_{\mathbb{C}}\rightarrow \mathfrak{g}_{\mathbb{C}}$ is diagonalizable, where $\mathfrak{g}_{\mathbb{C}}:=\mathfrak{g}+i\mathfrak{g}$ is the complexification of $\mathfrak{g}$.  Let $\mathfrak{t}\subset \mathfrak{g}$ be any maximal abelian Lie subalgebra.  Since $\mathfrak{t}$ is abelian, it follows that $\mbox{ad}_{H_1}\circ \mbox{ad}_{H_2}=\mbox{ad}_{H_2}\circ \mbox{ad}_{H_1}$ for all $H_1,H_2\in \mathfrak{t}$.  Hence, the set of linear maps
$$
\{\mbox{ad}_H: \mathfrak{g}_{\mathbb{C}}\rightarrow \mathfrak{g}_{\mathbb{C}}~|~H\in \mathfrak{t}_{\mathbb{C}}\}
$$
is simultaneously diagonalizable.  Consequently, $\mathfrak{g}_{\mathbb{C}}$ decomposes as
$$
\mathfrak{g}_{\mathbb{C}}=\mathfrak{t}_{\mathbb{C}}\oplus \bigoplus_{\alpha\in R}\mathfrak{g}_{\alpha}
$$
where $R\subset (\mathfrak{t}_\mathbb{C})^\ast\backslash\{0\}$ is a finite subset forming the so-called roots with respect to the choice of $\mathfrak{t}$ and 
$$
\mathfrak{g}_{\alpha}:=\{X\in \mathfrak{g}_{\mathbb{C}}~|~\mbox{ad}_HX=\alpha(H)X,~\forall~H\in \mathfrak{t}_{\mathbb{C}}\}.
$$
Since $\mathfrak{t}$ is a maximal abelian Lie subalgebra, we also see that $\mathfrak{g}_0=\mathfrak{t}_{\mathbb{C}}$, that is, $\mathfrak{t}_{\mathbb{C}}$ is the root space corresponding to the zero root. It is easy to verify that 
$$
[\mathfrak{g}_{\alpha},\mathfrak{g}_{\beta}]\subset \mathfrak{g}_{\alpha+\beta}
$$
if $\alpha+\beta\in R$.  Otherwise, $[\mathfrak{g}_{\alpha},\mathfrak{g}_{\beta}]=0$. It takes much more work to show that (1) $\dim \mathfrak{g}_{\alpha}=1$ for all $\alpha\in R$ and (2) if $\alpha\in R$, then the only other multiple of $\alpha$ is $-\alpha$.  

Key to the Samelson construction of integrable almost complex structures on even dimensional compact Lie groups is the fact that $R$ decomposes as
$$
R=\Delta^+\cup \Delta^-
$$ 
where $\Delta^+\cap \Delta^-=\emptyset$ and $\Delta^-=-\Delta^+$.  The properties of $\Delta^{\pm}$ are such that if $\alpha,\beta\in \Delta^+$ and $\alpha+\beta\in R$, then $\alpha+\beta\in \Delta^+$.  $\Delta^+$ and $\Delta^-$ are called the positive and negative roots of $R$ respectively.  From the above discussion, we can rewrite the root space decomposition as 
$$
\mathfrak{g}_{\mathbb{C}}=\mathfrak{t}_{\mathbb{C}}\oplus \bigoplus_{\alpha\in \Delta^+}(\mathfrak{g}_{\alpha}\oplus g_{-\alpha}).
$$
It is easy to see from the definition of $g_\alpha$ that $g_\alpha$ and $g_{-\alpha}$ are conjugates of one another:
$$
g_{-\alpha}=\overline{g_{\alpha}}.
$$
We refer the reader to the standard references on Lie theory (e.g. \cite{Var1984, Sep2007}) for proofs of the above facts.

\section{Main Result}
\begin{theorem}
    \label{ThmMainTheorem}
    For any compact Lie group $G$ (even or odd dimensional), the tangent Lie group $TG$ admits a left-invariant integrable almost complex structure.  More generally, $T^kG$, for $k\ge 1$, admits an integrable left-invariant almost complex structure, where $T^kG:=T(T^{k-1}G)$ and $T^0G:=G$.
\end{theorem}
\begin{proof}
Let $\mathfrak{t}$ be any maximal abelian Lie subalgebra of $\mathfrak{g}$ and let 
$$
\mathfrak{g}_{\mathbb{C}}=\mathfrak{t}_{\mathbb{C}}\oplus \bigoplus_{\alpha\in \Delta^+}(\mathfrak{g}_{\alpha}\oplus g_{-\alpha})
$$
be the root space decomposition associated to $\mathfrak{t}$ with $\Delta^+$ a system of positive roots. Let $H_1,\dots, H_k$ be a (real) basis of $\mathfrak{t}\subset \mathfrak{t}_{\mathbb{C}}$ and let $E_\alpha\in \mathfrak{g}_\alpha$, $E_{-\alpha}\in \mathfrak{g}_{-\alpha}$ be any nonzero elements.  Let $\mathcal{G}:=\mbox{Lie}(TG)$ and let $\mathcal{G}_{\mathbb{C}}:=\mathcal{G}+i\mathcal{G}$ be its complexification.  Then a basis of $\mathcal{G}_{\mathbb{C}}$ is obtained by taking the complete and vertical lifts of the aforementioned basis of $\mathfrak{g}_{\mathbb{C}}$.  Hence, $\mathcal{G}_{\mathbb{C}}$ has basis 
\begin{equation}
\label{eqBasisTG}
H_i^c,~E^c_{\alpha},~E^c_{-\alpha},~H_j^v,~E^v_{\alpha},~E^v_{-\alpha},~\hspace*{0.1in}1\le i,j\le k,~\alpha\in \Delta^+.
\end{equation}
We now construct a Samelson-type complex structure on $TG$.  Our notation follows the review of Samelson's construction given in section 2 of \cite{FG2023}.  Let $\widehat{J}:\mathcal{G}_{\mathbb{C}}\rightarrow \mathcal{G}_{\mathbb{C}}$ be the $\mathbb{C}$-linear map defined by
\begin{align}
    \label{eqJH}
    &\widehat{J}H_i^c:=H_i^v,\hspace*{0.1in}\widehat{J}H_i^v:=-H_i^c,\\
    \label{eqJE}
    \widehat{J}E_{\alpha}^c= iE^c_{\alpha},~\widehat{J}E_{\alpha}^v&= iE^v_{\alpha},~\widehat{J}E_{-\alpha}^c= -iE^c_{-\alpha},~\widehat{J}E_{-\alpha}^v= -iE^v_{-\alpha}.
\end{align}
From the definition, we immediately have $\widehat{J}^2=-id_{\mathcal{G}_{\mathbb{C}}}$. As was the case in Samelson \cite{Sam1953}, the definition of $\widehat{J}$ implies that $\widehat{J}\mathcal{G}\subset \mathcal{G}$.  In other words, $\widehat{J}:\mathcal{G}_{\mathbb{C}}\rightarrow \mathcal{G}_{\mathbb{C}}$ is induced by the $\mathbb{C}$-linear extension of the real linear map $\widehat{J}|_{\mathcal{G}}: \mathcal{G}\rightarrow \mathcal{G}$ (which, of course, also squares to negative the identity).

Indeed, from the definition, we already have $\widehat{J}\mathfrak{t}^c\subset \mathfrak{t}^v\subset \mathcal{G}$ and $\widehat{J}\mathfrak{t}^v\subset \mathfrak{t}^c\subset \mathcal{G}$ since $H_i\in \mathfrak{t}\subset \mathfrak{g}$.  Now, write $E^c_{\alpha}=X_\alpha^c+iY_\alpha^c$ with $X_\alpha,Y_\alpha\in \mathfrak{g}$.  Since $\dim \mathfrak{g}_{\pm \alpha}=1$ and $\mathfrak{g}_{\alpha}$ and $\mathfrak{g}_{-\alpha}$ are conjugate, we may assume $E_{-\alpha}=\overline{E_{\alpha}}$ without any loss of generality.  Hence, $E^c_{-\alpha}=X_\alpha^c-iY_\alpha^c$.  So the condition $\widehat{J}E^c_{\alpha}=iE^c_{\alpha}$ and $\widehat{J}E^c_{-\alpha}=-iE_{-\alpha}^c$ is equivalent to the requirement 
$$
\widehat{J}X_\alpha^c=-Y^c_\alpha,\hspace*{0.1in}\widehat{J}Y_\alpha^c=X^c_\alpha.
$$
Likewise, the condition $\widehat{J}E^v_{\alpha}=iE^v_{\alpha}$ and $\widehat{J}E^v_{-\alpha}=-iE_{-\alpha}^v$ implies 
$$
\widehat{J}X_\alpha^v=-Y^v_\alpha,\hspace*{0.1in}\widehat{J}Y_\alpha^v=X^v_\alpha.
$$
Since the elements 
$$
H_i^c,~X_\alpha^c,~Y_\alpha^c,~H^v_j,~X_{\alpha}^v,~Y_\alpha^v,~1\le i,j\le k,~\alpha\in \Delta^+
$$
form a basis of the (real) Lie algebra $\mathcal{G}:=\mbox{Lie}(TG)$, we see that $\widehat{J}\mathcal{G}\subset \mathcal{G}$. 

To show that $\widehat{J}|_{\mathcal{G}}:\mathcal{G}\rightarrow \mathcal{G}$ is integrable, it suffices to show that 
$$
N_{\widehat{J}}(A,B)=0,\hspace*{0.1in}\forall~A,B\in \mathcal{G}_{\mathbb{C}}.
$$
The latter is equivalent to verifying that $N_{\widehat{J}}(A,B)=0$ when $A,B$ runs over all possible combinations of the basis elements in (\ref{eqBasisTG}).  We rely heavily on the bracket relations of complete and vertical lifts (\ref{EqTGcvRules}) in verifying all possible cases.
\newline\\
\noindent \textbf{case 1:} $A=H_i^c$, $B=E_\alpha^c$
\begin{align*}
    N_{\widehat{J}}(H_i^c,E_\alpha^c)&=\widehat{J}[\widehat{J}H_i^c,E_\alpha^c]+\widehat{J}[H^c_i,\widehat{J}E^c_\alpha]+[H^c_i,E^c_\alpha]-[\widehat{J}H^c_i,\widehat{J}E^c_\alpha]\\
    &=\widehat{J}[H_i^v,E_\alpha^c]+i\widehat{J}[H^c_i,E^c_\alpha]+[H_i,E_\alpha]^c-i[H^v_i,E^c_\alpha]\\
    &=\widehat{J}[H_i,E_\alpha]^v+i\widehat{J}[H_i,E_\alpha]^c+\alpha(H_i)E_\alpha^c-i[H_i,E_\alpha]^v\\
    &=\alpha(H_i)\widehat{J}E_\alpha^v+i\alpha(H_i)\widehat{J}E_\alpha^c+\alpha(H_i)E_\alpha^c-i  \alpha(H_i)E_\alpha^v\\
    &=i\alpha(H_i)E_\alpha^v+(i)(i)\alpha(H_i)E_\alpha^c+\alpha(H_i)E_\alpha^c-i  \alpha(H_i)E_\alpha^v\\
    &=0.
\end{align*}

\noindent \textbf{case 2:} $A=H_i^c$, $B=E_{-\alpha}^c$
\begin{align*}
    N_{\widehat{J}}(H_i^c,E_{-\alpha}^c)&=\widehat{J}[\widehat{J}H_i^c,E_{-\alpha}^c]+\widehat{J}[H^c_i,\widehat{J}E^c_{-\alpha}]+[H^c_i,E^c_{-\alpha}]-[\widehat{J}H^c_i,\widehat{J}E^c_{-\alpha}]\\
    &=\widehat{J}[H_i^v,E_{-\alpha}^c]-i\widehat{J}[H^c_i,E^c_{-\alpha}]+[H_i,E_{-\alpha}]^c+i[H^v_i,E^c_{-\alpha}]\\
    &=\widehat{J}[H_i,E_{-\alpha}]^v-i\widehat{J}[H_i,E_{-\alpha}]^c-\alpha(H_i)E_{-\alpha}^c+i[H_i,E_{-\alpha}]^v\\
    &=-\alpha(H_i)\widehat{J}E_{-\alpha}^v+i\alpha(H_i)\widehat{J}E_{-\alpha}^c-\alpha(H_i)E_{-\alpha}^c-i\alpha(H_i)E_{-\alpha}^v\\
    &=i\alpha(H_i)E_{-\alpha}^v+i(-i)\alpha(H_i)E_{-\alpha}^c-\alpha(H_i)E_{-\alpha}^c-i\alpha(H_i)E_{-\alpha}^v\\
    &=0.
\end{align*}

\noindent \textbf{case 3:} $A=H_i^c$, $B=E_\alpha^v$
\begin{align*}
N_{\widehat{J}}(H_i^c,E_\alpha^v)&=\widehat{J}[\widehat{J}H_i^c,E_\alpha^v]+\widehat{J}[H_i^c,\widehat{J}E_\alpha^v]+[H_i^c,E_\alpha^v]-[\widehat{J}H_i^c,\widehat{J}E^v_\alpha]\\
&=\widehat{J}[H_i^v,E_\alpha^v]+i\widehat{J}[H_i^c,E_\alpha^v]+[H_i,E_\alpha]^v-i[H_i^v,E^v_\alpha]\\
&=i\widehat{J}[H_i,E_\alpha]^v+\alpha(H_i)E_\alpha^v\\
&=i\alpha(H_i)\widehat{J}E_\alpha^v+\alpha(H_i)E_\alpha^v\\
&=(i)(i)\alpha(H_i)E_\alpha^v+\alpha(H_i)E_\alpha^v\\
&=0.
\end{align*}

\noindent \textbf{case 4:} $A=H_i^c$, $B=E_{-\alpha}^v$
\begin{align*}
N_{\widehat{J}}(H_i^c,E_{-\alpha}^v)&=\widehat{J}[\widehat{J}H_i^c,E_{-\alpha}^v]+\widehat{J}[H_i^c,\widehat{J}E_{-\alpha}^v]+[H_i^c,E_{-\alpha}^v]-[\widehat{J}H_i^c,\widehat{J}E^v_{-\alpha}]\\
&=\widehat{J}[H_i^v,E_{-\alpha}^v]-i\widehat{J}[H_i^c,E_{-\alpha}^v]+[H_i,E_{-\alpha}]^v+i[H_i^v,E^v_{-\alpha}]\\
&=-i\widehat{J}[H_i,E_{-\alpha}]^v-\alpha(H_i)E_{-\alpha}^v\\
&=i\alpha(H_i)\widehat{J}E_{-\alpha}^v-\alpha(H_i)E_{-\alpha}^v\\
&=i(-i)\alpha(H_i)E_{-\alpha}^v-\alpha(H_i)E_{-\alpha}^v\\
&=0.
\end{align*}

\noindent \textbf{case 5:} $A=H_i^v$, $B=E_\alpha^c$
\begin{align*}
    N_{\widehat{J}}(H_i^v,E_\alpha^c)=N_{\widehat{J}}(\widehat{J}H_i^c,E_\alpha^c)=-\widehat{J}N_{\widehat{J}}(H^c_i,E_\alpha^c)=0
\end{align*}

\noindent \textbf{case 6:} $A=H_i^v$, $B=E_{-\alpha}^c$
\begin{align*}
    N_{\widehat{J}}(H_i^v,E_{-\alpha}^c)=N_{\widehat{J}}(\widehat{J}H_i^c,E_{-\alpha}^c)=-\widehat{J}N_{\widehat{J}}(H^c_i,E_{-\alpha}^c)=0
\end{align*}

\noindent \textbf{case 7:} $A=H_i^v$, $B=E_\alpha^v$
\begin{align*}
    N_{\widehat{J}}(H_i^v,E_{\alpha}^v)=N_{\widehat{J}}(\widehat{J}H_i^c,E_{\alpha}^v)=-\widehat{J}N_{\widehat{J}}(H^c_i,E_{\alpha}^v)=0
\end{align*}
\noindent \textbf{case 8:} $A=H_i^v$, $B=E_{-\alpha}^v$
\begin{align*}
    N_{\widehat{J}}(H_i^v,E_{-\alpha}^v)=N_{\widehat{J}}(\widehat{J}H_i^c,E_{-\alpha}^v)=-\widehat{J}N_{\widehat{J}}(H^c_i,E_{-\alpha}^v)=0
\end{align*}

\noindent \textbf{case 9:} $A=E^c_\alpha$, $B=E^c_\beta$ or $A=E^c_\alpha$, $B=E^v_\beta$,~$\alpha,\beta\in \Delta^+$\newline
For this case, if $\alpha+\beta$ is not a root, then $[\mathfrak{g}_\alpha,\mathfrak{g}_\beta]=0$.  This combined with the bracket relations (\ref{EqTGcvRules}) imply $N_{\widehat{J}}(A,B)=0$.  So let us assume that $\alpha+\beta$ is a root and since $\alpha,~\beta\in \Delta^+$, it follows that $\alpha+\beta\in \Delta^+$ as well.  This gives the following:
\begin{align*}
    N_{\widehat{J}}(A,B)&=\widehat{J}[\widehat{J}A, B]+\widehat{J}[A,\widehat{J}B]+[A,B]-[\widehat{J}A,\widehat{J}B]\\
    &=i\widehat{J}[A, B]+i\widehat{J}[A,B]+[A,B]-(i)^2[A,B]\\
    &=i(i)[A, B]+i(i)[A,B]+[A,B]-(i)^2[A,B]\\
    &=0
\end{align*}
where in the third equality we use the fact that $\alpha+\beta$ is a positive root which implies that $[A,B]$ is a lift of an element of $\mathfrak{g}_{\alpha+\beta}$.\newline

\noindent \textbf{case 10:} $A=E_\alpha^v$, $B=E_\beta^v$ where $\alpha,\beta\in \Delta^+\cup (-\Delta^+)$\newline
For this case, we immediately have $N_{\widehat{J}}(E_\alpha^v,E_\beta^v)=0$ since the Lie bracket of two vertical lifts is always zero and $\widehat{J}E^v_\alpha$ is still vertical.\newline 

\noindent \textbf{case 11:} $A=E^c_\alpha$, $B=E^c_{-\beta}$ or $A=E^c_\alpha$, $B=E^v_{-\beta}$ or $A=E^v_\alpha$, $B=E^c_{-\beta}$,~$\alpha,\beta\in \Delta^+$
\begin{align*}
    N_{\widehat{J}}(A,B)&=\widehat{J}[\widehat{J}A,B]+\widehat{J}[A,\widehat{J}B]+[A,B]-[\widehat{J}A,\widehat{J}B]\\
    &=(i)\widehat{J}[A,B]-i\widehat{J}[A,B]+[A,B]-(i)(-i)[A,B]\\
    &=0.
\end{align*}

\noindent \textbf{case 12:} $A=E^c_{-\alpha}$, $B=E^c_{-\beta}$ or $A=E^c_{-\alpha}$, $B=E^v_{-\beta}$ ~$\alpha,\beta\in \Delta^+$\newline
If $-\alpha-\beta$ is not a root, then $[\mathfrak{g}_{-\alpha},\mathfrak{g}_{-\beta}]=0$.  The bracket relations (\ref{EqTGcvRules}) and the definition of $\widehat{J}$ then implies that $N_{\widehat{J}}(A,B)=0$.  So let us assume that $-\alpha-\beta$ is a root.  This gives the following:
 \begin{align*}
     N_{\widehat{J}}(A,B)&=\widehat{J}[\widehat{J}A,B]+\widehat{J}[A,\widehat{J}B]+[A,B]-[\widehat{J}A,\widehat{J}B]\\
     &=(-i)\widehat{J}[A,B]-i\widehat{J}[A,B]+[A,B]-(-i)^2[A,B]\\
     &=(-i)^2[A,B]-i(-i)[A,B]+[A,B]-(-i)^2[A,B]\\
     &=0,
 \end{align*}
where in the third equality, we have used the fact that $-\alpha-\beta$ is a negative root which implies that $[A,B]$ is a lift of an element of $\mathfrak{g}_{-\alpha-\beta}$. \newline

\noindent \textbf{case 13:} $A\in\{H^c_i, H^v_i\}$, $B\in \{H^c_j,H^v_j\}$
\newline
Since $\mathfrak{t}$ is abelian, we see immediately from the definition of $\widehat{J}$ and the bracket relations (\ref{EqTGcvRules}) that $N_{\widehat{J}}(A,B)=0$ for this case.\newline

Cases 1-13 above show that $N_{\widehat{J}}=0$.  The statement that the iterated tangent Lie group $T^kG:=T(T^{k-1}G)$ admits left-invariant integrable almost complex structures for $k\ge 1$ is a consequence of the fact that $TG$ admits a left-invariant integrable almost complex structure (as shown above) and Proposition \ref{propTGTower} below.  This completes the proof.
\end{proof}
\begin{proposition}
\label{propTGTower}
Let $G$ be any Lie group with a left-invariant integrable almost complex structure $J$.  Then the tangent Lie group $TG$ also admits a left-invariant integrable almost complex structure $J$.
\end{proposition}
\begin{proof}
Let $\mathcal{G}:=\mbox{Lie}(TG)$.  Define $\widehat{J}:\mathcal{G}\rightarrow \mathcal{G}$ via 
$$
\widehat{J}X^c:=(JX)^c,\hspace*{0.2in} \widehat{J}X^v:=(JX)^v.
$$
From (\ref{EqTGcvRules}) and the definition of $\widehat{J}$, it immediately follows that $N_{\widehat{J}}(X^v,Y^v)=0$ and
$$
N_{\widehat{J}}(X^c,Y^c)=N_J(X,Y)^c,\hspace*{0.1in}N_{\widehat{J}}(X^c,Y^v)=N_{J}(X,Y)^v.
$$
Since $N_{J}=0$, we see that $\widehat{J}$ is integrable.
\end{proof}

\begin{remark}  
If $G$ is already an even dimensional compact Lie group, one can construct a left-invariant integrable almost complex structure on $G$ via Samelson's construction \cite{Sam1953} (c.f. Sec. 2 of \cite{FG2023}) and then lift it to $TG$ via Proposition \ref{propTGTower}.  This provides an alternate means of equipping $TG$ with a left-invariant integrable almost complex structure.  However, if $G$ is a compact odd dimensional Lie group, then $G$ admits no almost complex structure (let alone a left-invariant integrable one) that could be lifted to $TG$.  In this case, one can use the construction in the proof of Theorem \ref{ThmMainTheorem} to equip the tangent Lie group $TG$ with a left-invariant integrable almost complex structure.  Once $TG$ is equipped with such a structure, one obtains an infinite tower of non-compact Lie groups 
$$
TG\subset T^2G\subset \cdots \subset T^kG \subset \cdots
$$
which all admit left-invariant integrable almost complex structures.
\end{remark}

\section{Examples}
\begin{example}
The Lie algebra $\mathfrak{so}(3)=\mbox{Lie}(SO(3))$ has basis $e_1,e_2,e_3$ with nonzero bracket relations
$$
[e_1,e_2]=-e_3,~[e_1,e_3]=e_2,~[e_2,e_3]=-e_1.
$$
Let $\mathfrak{t}\subset \mathfrak{so}(3)$ be the (real) abelian Lie subalgebra spanned by $e_1$.  Then the root space decomposition of $\mathfrak{so}(3)_{\mathbb{C}}$ with respect to $\mathfrak{t}$ is 
$$
\mathfrak{so}(3)_{\mathbb{C}}=\mathfrak{t}_{\mathbb{C}}\oplus \mathfrak{so}(3)_\alpha\oplus \mathfrak{so}(3)_{-\alpha}
$$
where 
$$
\mathfrak{so}(3)_\alpha=\mbox{span}_{\mathbb{C}}\{e_2+ie_3\}
$$
and 
$$
\mathfrak{so}(3)_{-\alpha}=\mbox{span}_{\mathbb{C}}\{e_2-ie_3\}
$$
and the root $\alpha: \mathfrak{t}_{\mathbb{C}}\rightarrow \mathbb{C}$ is defined by $\alpha(e_1)=i$. We take $\Delta^+=\{\alpha\}$ as our positive root system.

Let $E_\alpha:=e_2+ie_3$ and $E_{-\alpha}:=e_2-ie_3$. Let $\mathcal{G}:=\mbox{Lie}(TSO(3))$.  For our basis of $\mathcal{G}_{\mathbb{C}}$, we take
$$
e_1^c,~E_\alpha^c,~E_{-\alpha}^c,~e_1^v,~E_\alpha^v,~E_{-\alpha}^v.
$$
From the proof of Theorem \ref{ThmMainTheorem}, the Samelson complex structure on $TSO(3)$ is the linear map $\widehat{J}:\mathcal{G}\rightarrow \mathcal{G}$ whose $\mathbb{C}$-linear extension satisfies
$$
\widehat{J}e_1^c=e_1^v,~\widehat{J}E^c_{\alpha}=iE^c_{\alpha},~\widehat{J}E^c_{-\alpha}=-iE^c_{-\alpha},
$$
$$
\widehat{J}e_1^v=-e_1^c,~\widehat{J}E^v_{\alpha}=iE^v_{\alpha},~\widehat{J}E^v_{-\alpha}=-iE^v_{-\alpha}.
$$
In terms of the (real) basis on $\mathcal{G}$ given by
$$
e_1^c,~e_2^c,~e_3^c,~e_1^v,~e_2^v,~e_3^v,
$$
we have
$$
\widehat{J}e_1^c=e_1^v,~\widehat{J}e^c_2=-e_3^c,~\widehat{J}e^c_3=e^c_2,
$$
$$
\widehat{J}e_1^v=-e_1^c,~\widehat{J}e^v_2=-e_3^v,~\widehat{J}e^v_3=e^v_2.
$$
\end{example}

\begin{example}
    The Lie algebra $\mathfrak{u}(3):=\mbox{Lie}(U(3))$ has basis $e_1,e_2,\dots, e_9$ with nonzero bracket relations
    $$
    [e_1,e_4]=e_7,~[e_1,e_5]=e_8,~[e_1,e_7]=-e_4,[e_1,e_8]=-e_5,~[e_2,e_4]=-e_7,
    $$
    $$
    [e_2,e_6]=e_9,~[e_2,e_7]=e_4,~[e_2,e_9]=-e_6,~[e_3,e_5]=-e_8,~[e_3,e_6]=-e_9,
    $$
    $$
    [e_3,e_8]=e_5,~[e_3,e_9]=e_6,~[e_4,e_5]=-e_6,~[e_4,e_6]=e_5,~[e_4,e_7]=2e_1-2e_2,
    $$
    $$
    [e_4,e_8]=-e_9,~[e_4,e_9]=e_8,~[e_5,e_6]=-e_4,~[e_5,e_7]=-e_9,~[e_5,e_8]=2e_1-2e_3,
    $$
    $$
    [e_5,e_9]=e_7,~[e_6,e_7]=-e_8,~[e_6,e_8]=e_7,~[e_6,e_9]=2e_2-2e_3,
    $$
    $$
    [e_7,e_8]=-e_6,~[e_7,e_9]=-e_5,~[e_8,e_9]=-e_4.
    $$
    A maximal abelian Lie subalgebra of $\mathfrak{u}(3)$ is $\mathfrak{t}=\{e_1,e_2,e_3\}$. The root space decomposition of $\mathfrak{u}(3)_{\mathbb{C}}$ is 
    \begin{align*}
        \mathfrak{u}(3)_{\mathbb{C}}=\mathfrak{t}_{\mathbb{C}}\oplus \mathfrak{u}(3)_{\alpha}\oplus \mathfrak{u}(3)_{-\alpha}\oplus \mathfrak{u}(3)_{\beta}\oplus \mathfrak{u}(3)_{-\beta}\oplus \mathfrak{u}(3)_{\gamma}\oplus \mathfrak{u}(3)_{-\gamma}
    \end{align*}
    where
    $$
        \mathfrak{u}(3)_{\alpha}=\mbox{span}_{\mathbb{C}}\{e_4-ie_7\},
    $$
    $$
        \mathfrak{u}(3)_{\beta}=\mbox{span}_{\mathbb{C}}\{e_5+ie_8\},
    $$
    $$
        \mathfrak{u}(3)_{\gamma}=\mbox{span}_{\mathbb{C}}\{e_6+ie_9\},
    $$
    with positive root system $\Delta^+=\{\alpha,\beta,\gamma\}$ given by
    $$
        \alpha(e_1)=i,~\alpha(e_2)=-i,~\alpha(e_3)=0,
    $$
    $$
        \beta(e_1)=-i,~\beta(e_2)=0,~\beta(e_3)=i,
    $$
    $$
        \gamma(e_1)=0,~\gamma(e_2)=-i,~\gamma(e_3)=i.
    $$
    Let $\mathcal{G}:=\mbox{Lie}(TU(3))$. A basis for $\mathcal{G}_{\mathbb{C}}$ is 
    $$
e_1^c,~e_2^c,~e_3^c,~E_\alpha^c,~E^c_\beta,~E^c_\gamma,E_{-\alpha}^c,~E^c_{-\beta},~E^c_{-\gamma},
    $$
    $$
~e_1^v,~e_2^v,~e_3^v,~E_\alpha^v,~E^v_\beta,~E^v_\gamma,~E_{-\alpha}^v,~E^v_{-\beta},~E^v_{-\gamma},
    $$
    where
    $$
    E_\alpha=e_4-ie_7,~E_\beta=e_5+ie_8,~E_\gamma=e_6+ie_9,
    $$
    $$
    E_{-\alpha}=e_4+ie_7,~E_{-\beta}=e_5-ie_8,~E_{-\gamma}=e_6-ie_9.
    $$
    From the proof of Theorem \ref{ThmMainTheorem}, we obtain an integrable left-invariant almost complex structure by defining $\widehat{J}:\mathcal{G}\rightarrow \mathcal{G}$ to be the linear map whose $\mathbb{C}$-linear extension satisfies
    $$
    \widehat{J}e_1^c=e_1^v,~\widehat{J}e_2^c=e_2^v,~\widehat{J}e_3^c=e_3^v,
    $$
    $$
    \widehat{J}E^c_\alpha=iE^c_\alpha,~\widehat{J}E^c_{\beta}=iE^c_{\beta},~\widehat{J}E^c_{\gamma}=iE^c_{\gamma},
    $$
    $$
    \widehat{J}E^c_{-\alpha}=-iE^c_{-\alpha},~\widehat{J}E^c_{-\beta}=-iE^c_{-\beta},~\widehat{J}E^c_{-\gamma}=-iE^c_{-\gamma},
    $$
    $$
    \widehat{J}e_1^v=-e_1^c,~\widehat{J}e_2^v=-e_2^c,~\widehat{J}e_3^v=-e_3^c,
    $$
    $$
    \widehat{J}E^v_\alpha=iE^v_\alpha,~\widehat{J}E^v_{\beta}=iE^v_{\beta},~\widehat{J}E^v_{\gamma}=iE^v_{\gamma},
    $$
    $$
    \widehat{J}E^v_{-\alpha}=-iE^v_{-\alpha},~\widehat{J}E^v_{-\beta}=-iE^v_{-\beta},~\widehat{J}E^v_{-\gamma}=-iE^v_{-\gamma}.
    $$
    Expressing $\widehat{J}$ in terms of the 
    (real) basis of $\mathcal{G}$ given by
    $$
    e_1^c,~e_2^c,~\dots, e_9^c,~e_1^v,~e_2^v,~\dots, e_9^v,
    $$
    we have 
    $$
    \widehat{J}e_1^c=e_1^v,~\widehat{J}e_2^c=e_2^v,~\widehat{J}e_3^c=e_3^v,
    $$
    $$
    \widehat{J}e_4^c=e_7^c,~\widehat{J}e_5^c=-e^c_8,~\widehat{J}e^c_6=-e^c_9,
    $$
    $$
    \widehat{J}e^c_7=-e^c_4,~\widehat{J}e^c_8=e^c_5,~\widehat{J}e^c_9=e^c_6,
    $$
    $$
    \widehat{J}e_1^v=-e_1^c,~\widehat{J}e_2^v=-e_2^c,~\widehat{J}e_3^v=-e_3^c,
    $$
    $$
    \widehat{J}e_4^v=e_7^v,~\widehat{J}e_5^v=-e^v_8,~\widehat{J}e^v_6=-e^v_9,
    $$
    $$
    \widehat{J}e^v_7=-e^v_4,~\widehat{J}e^v_8=e^v_5,~\widehat{J}e^v_9=e^v_6.
    $$
\end{example}

\end{document}